\documentclass[12pt,oneside]{amsart}

\usepackage[utf8]{inputenc}
\usepackage[english]{babel}
\usepackage[margin=3cm]{geometry} 
\usepackage{amsmath,amsthm,amssymb}
\usepackage{mathtools}
\usepackage{graphicx}
\usepackage[colorlinks=true, allcolors=blue]{hyperref}
\usepackage{mathrsfs,calligra}
\usepackage{tikz-cd,adjustbox}
\usepackage{spverbatim,comment}
\usepackage[style=numeric,backend=biber]{biblatex}
\usepackage[shortlabels]{enumitem}

\theoremstyle{plain}
\newtheorem{theorem}{Theorem}[section]
\newtheorem{lemma}[theorem]{Lemma}
\newtheorem{proposition}[theorem]{Proposition}
\newtheorem{corollary}[theorem]{Corollary}

\theoremstyle{definition}
\newtheorem{definition}[theorem]{Definition}
\newtheorem{example}[theorem]{Example}

\theoremstyle{remark}
\newtheorem{remark}[theorem]{Remark}

\newtheorem{observations}[theorem]{Observations}

\newcommand{\Rom}[1]{\uppercase\expandafter{\romannumeral#1}}

\newcommand{\mc}{\mathcal}
\newcommand{\ms}{\mathscr}

\newcommand{\E}{\mathcal E}

\renewcommand{\H}{\mathcal H}

\newcommand{\CC}{\mathbb C}
\newcommand{\DD}{\mathbb D}

\newcommand{\LL}{\mathbb L}

\newcommand{\QQ}{\mathbb Q}
\newcommand{\RR}{\mathbb R}

\newcommand{\ZZ}{\mathbb Z}

\DeclareMathOperator{\im}{Im}

\DeclareMathOperator{\coker}{coker}

\DeclareMathOperator{\Hom}{Hom}
\DeclareMathOperator{\sHom}{\mathscr{H}\text{\kern -3pt {\calligra\large om}}\,}
\newcommand{\RHom}{{\bf R} \H om}
\DeclareMathOperator{\Ext}{Ext}
\newcommand{\sExt}{\E xt} %sheaf Ext

\newcommand{\DB}{\underline{\Omega}} %for Du Bois complex
\newcommand{\dt}{\otimes^{\LL}} %for derived tensor product

\DeclareMathOperator{\Cone}{Cone}
\DeclareMathOperator{\sing}{sing}
\DeclareMathOperator{\supp}{supp}

\DeclareMathOperator{\Ch}{ch}
\DeclareMathOperator{\Coh}{Coh}
\DeclareMathOperator{\O0}{\DB_X^0}
\DeclareMathOperator{\Length}{length}

\DeclareMathOperator{\lcm}{lcm}

\DeclarePairedDelimiter\abs{\lvert}{\rvert}

\makeatletter
\let\oldabs\abs
\def\abs{\@ifstar{\oldabs}{\oldabs*}}
\makeatother

\theoremstyle{definition} % just in case the style had changed
\newcommand{\thistheoremname}{}
\newtheorem*{genericthm}{\thistheoremname}

\newcommand{\Addresses}{{% additional braces for segregating \footnotesize
  \vspace{\bigskipamount}
  \footnotesize

  \textsc{Department of Mathematics, MIT, 77 Massachusetts Avenue, Cambridge, MA 02139, USA}\par\nopagebreak
  \textit{E-mail address}: 
  \texttt{annelars@mit.edu}
  
  \vspace{\bigskipamount}
  
  \textsc{Department of Mathematics, Harvard University, 1 Oxford Street, Cambridge, MA 02138, USA}\par\nopagebreak
  \textit{E-mail address}: \texttt{astenie@math.harvard.edu}
}}

\title{Reider-type theorems on normal surfaces via Bridgeland stability}
\author{Anne Larsen}
\author{Anda Tenie}
\thanks{A.L. was partially supported by the National Science Foundation Graduate Research Fellowship under Grant No. 2141064.}

\bibliography{main.bib}
\begin{document}

\maketitle

\begin{abstract}
Using Langer's construction of Bridgeland stability conditions on normal surfaces, we prove Reider-type theorems generalizing the work done by Arcara-Bertram in the smooth case. 
Our results still hold in positive characteristic or when $\omega_X\otimes L$ is not necessarily a line bundle. They also hold when the dualizing sheaf is replaced by a variant arising from the theory of Du Bois complexes. For complex surfaces with at most rational double point singularities, we recover the optimal bounds for global generation and very ampleness as predicted by Fujita's conjecture.

\end{abstract}

\section{Introduction}In \cite{fujita1988contribution} Fujita proposed a fundamental conjecture in algebraic geometry stating that if $L$ is an ample line bundle on a smooth projective complex variety $X$ of dimension $n$, then $\omega_X\otimes L^{\otimes n+1}$ is globally generated and $\omega_X\otimes L^{\otimes n+2}$ is very ample. The conjecture has seen much interest over the years and the global generation part is now known in dimension $\leq 5$ (\cite{reider1988vector}, \cite{ein1993global}, \cite{kawamata1995fujita}, \cite{ye2020fujita}).

One strategy that has seen much success (including in giving effective bounds for higher dimensional varieties) involves vanishing theorems. Another approach, which is, however, hard to generalize in higher dimensions, is Reider's strategy of using vector bundle techniques and Bogomolov's inequality. We will focus on the latter approach in the context of normal surfaces.
Our main source of inspiration is Arcara and Bertram's reinterpretation of Reider's theorem in the context of Bridgeland stability conditions \cite{arcara2009reider}.

Our work is concerned with giving Reider-type (and generation of higher jets of $\omega_X\otimes L^{\otimes a}$) bounds on normal surfaces.  For this generalization to singular surfaces, another key ingredient used in our work is Langer's recent construction of stability conditions on normal surfaces \cite{langer2024bridgeland}. 

\begin{remark}
Other Reider-type bounds for normal surfaces under the assumption that $K_X+H$ is Cartier can be found in \cite{sakai2006reider}, \cite{ein1995global}.  For generation of higher jets results when $X$ has in addition at most rational singularities one can consult \cite{langer1998note}.
\end{remark}

We now discuss the main results of this article.
% They are consequences of our main Theorem \ref{thm: generalvanishing}.
Let $X$ be a projective normal surface defined over an algebraically closed field (of any characteristic), $Z \subset X$ a dimension 0 subscheme of length $l_Z$, and let $L$ be an ample line bundle on $X$. The constant $C_X$ will be defined in Section \ref{sec:preliminaries} and in characteristic 0, depends only on the singularities of $X$; roughly speaking, $C_X$ ensures a Bogomolov inequality holds on $X.$  For concrete examples where this constant is computed see Section \ref{sec: cx}.

\begin{theorem}\label{thm :VanishingOmega}
Let $X$ be a projective normal surface defined over an algebraically closed field, $Z \subset X$ a zero-dimensional subscheme of length $l_Z$, and $L$ an ample line bundle on $X$. Suppose there is a pair of nonnegative integers $l_1 + l_2 = l_Z$ such that
$L^2 \ge \max( (C_X + 2l_2+1)^2, 4(l_1+l_2+C_X)+\epsilon) $
and $L \cdot C \ge \max(2(2l_1 + C_X), C_X + 2l_2 + 1)$ for any curve $C$.
Then we have the vanishing
$$H^1(X, \omega_X \dt L \otimes I_Z) = 0.$$ In particular, this implies $\omega_X \otimes L$ separates jets along $Z,$ i.e. 
$$H^0(\omega_X\otimes L)\rightarrow H^0(\omega_X\otimes L\otimes \mc O_Z)$$ is surjective.
\end{theorem}
 
 \begin{remark}
     
 For a ``Reider-type" statement see Theorem \ref{thm: Reider}. This variant will also recover in Corollary \ref{cor: Reidervample} the very ampleness part of Reider's theorem for smooth surfaces when $L$ is assumed to be an ample line bundle.
 
 \end{remark}
 
 \begin{definition}\label{def: m'}
    Define a function $m: \RR_{\ge 0} \times \ZZ_{\ge 0} \to \RR$ by
    $$(C_X, l) \mapsto \min_{\substack{l_1+l_2 = l \\ l_1, l_2 \in \ZZ_{\ge 0}}} \{\max\left(2(2l_1+C_X), C_X + 2l_2+1\right)\}.$$
    Also define
    $$m'(C_X, l) := \begin{cases}
        3 & C_X = 1, l = 0 \\
        m(C_X, l) & \textrm{else}
    \end{cases}$$
\end{definition}

Note, in particular,  that $m'(0, 1)=3$ and $m'(0,2) =4$ which will recover the global generation and very ampleness bounds in Fujita's conjecture for complex surfaces with at worst rational double point singularities.

\begin{corollary}\label{cor: abound}
Let $X$ be a projective normal surface defined over an algebraically closed field, $Z \subset X$ a subscheme of length $l_Z$, and $L$ an ample line bundle on $X$. Then
$$H^1(X, \omega_X \dt L^{\otimes a} \otimes I_Z) = 0$$
for any integer
$ a \ge m'(C_X, l_Z).$ In particular, this implies $\omega_X \otimes L^{\otimes a}$ separates jets along $Z,$ i.e. 
$$H^0(\omega_X \otimes L^{\otimes a})\rightarrow H^0(\omega_X \otimes L^{\otimes a}\otimes \mc O_Z)$$ is surjective.
\end{corollary}

\begin{remark}
     Unlike previous work, which gives similar results for a divisor $H$ under the assumption that $K_X + H$ is Cartier, our theorem shows separation of jets for the tensor product of an ample bundle with the canonical sheaf, even when this is not a line bundle.
\end{remark} 

We also recover the classical Fujita conjecture for surfaces:
\begin{corollary}\label{cor: intro cx=0 fujita}
    If $X$ is a projective normal surface with $C_X = 0$ (e.g. a complex surface with at most rational double point singularities) and $L$ an ample line bundle, then $\omega_X \otimes L^{\otimes a}$ is globally generated for $a \ge 3$ and very ample for $a \ge 4$.
\end{corollary}

\begin{remark}
  The proof of the optimal bound for very ampleness in the smooth case is contained in \cite{arcara2009reider}, but our argument appears to be the first proof of the correct bound for global generation via Bridgeland stability. 
\end{remark} 

In positive characteristic, we can no longer take $C_X = 0$ for every smooth surface $X$, but there is still an explicit bound:

\begin{corollary}
    Let $X$ be a smooth projective surface defined over an algebraically closed field of positive characteristic and $Z \subset X$ a dimension 0 subscheme of length $l_Z.$ Following Koseki \cite{koseki2023bogomolov}, $C_X$ depends only on the birational equivalence class of $X$ and is defined as: 
\begin{enumerate}
     \item  If $X$ is a minimal surface of general type, then $C_X = 2 + 5K_X^2 -\chi(\mc O_X).$
    \item  If $\kappa(X) = 1$ and $X$ is quasi-elliptic, then $C_X = 2 - \chi(\mc O_X).$
    \item Otherwise, $C_X = 0.$
    \end{enumerate} Then
$$H^1(X, \omega_X \dt L^{\otimes a} \otimes I_Z) = 0$$
for any integer
$ a \ge m'(C_X, l_Z).$ In particular, this implies $\omega_X \otimes L^{\otimes a}$ separates jets along $Z,$ i.e. 
$$H^0(\omega_X \otimes L^{\otimes a})\rightarrow H^0(\omega_X \otimes L^{\otimes a}\otimes \mc O_Z)$$ is surjective. 
\end{corollary}

\begin{remark}
It is known that Fujita's conjecture does not hold for smooth surfaces in positive characteristic. In fact, for any power $a$ one can find a smooth surface $X$ such that $\omega_X\otimes L^{\otimes a}$ is not globally generated \cite{gu2022counterexamples}. Our results account for this, as the bound on $a$ depends on the constant $C_X$ corresponding to the surface $X.$ 
\end{remark}
\begin{remark}
Note that for $l_Z=0$ we get a Kodaira vanishing type result in positive characteristic, again depending on $C_X.$ 
\end{remark}

Now specializing to the case when $X$ is complex, we also obtain similar results when the dualizing sheaf $\omega_X$ is replaced by the dual of $\O0,$ where $\O0$ is the $0$-th Du Bois complex of $X$. This line of inquiry was based on the observation that $\O0[1]$ has the shape of a (potentially) Bridgeland stable object (namely, slope stable sheaf in degree $-1$ and dimension 0 sheaf in degree 0), and follows the philosophy of viewing the Du Bois complexes as better-behaved versions of the sheaves of K\"ahler differentials on $X$ (or in this case, of $\O0$ as a better-behaved version of $\ms O_X$).
%To be more precise, $\O0$  is 0-th associated graded term of the filtered de Rham complex $\Omega_X^\bullet$ with respect to the Hodge filtration.
To be more precise, consider  $$\omega_X^{DB}:= \RHom(\O0, \omega_X),$$ which is in fact a sheaf (see Section \ref{sec:DB} for details). Then we obtain similar Reider-type results for $\omega_X^{DB}\otimes L^{\otimes a}.$ (It may be helpful to keep in mind that, at least in the Gorenstein case, $\omega_X^{DB} = \omega_X \otimes I$, where $I$ is an ideal sheaf supported exactly on the non-Du Bois locus. Thus, the problem of $\omega_X^{DB} \otimes L^{\otimes a}$ separating jets along some zero-dimensional subscheme $Z$ is more or less the problem of $\omega_X \otimes L^{\otimes a}$ separating jets along a larger subscheme containing both $Z$ and the non-Du Bois locus, and so the bounds are correspondingly worse.)

\begin{theorem}\label{thm: DB+RD}
Let $X$ be a projective complex normal surface, $Z \subset X$ a subscheme of length $l_Z$, and $L$ an ample line bundle on $X$. Let $l_T := \Length(\mc H^1(\O0))$. 
Let $l_1, l_2$ be a pair of nonnegative integers such that $l_Z + l_T = l_1 + l_2$.
Suppose
$L^2 \ge \max( (C_X + 2l_2+1)^2, 4(l_1+l_2+C_X) + \epsilon) $
and $L \cdot C \ge \max(2(2l_1 + C_X), C_X + 2l_2 + 1)$ for every curve $C$.
Then we have the vanishing
$$H^1(X, \omega_X^{DB} \dt L \otimes I_Z) = 0.$$ In particular, this implies $\omega_X^{DB} \otimes L$ separates jets along $Z,$ i.e. 
$$H^0(\omega_X^{DB}\otimes L)\rightarrow H^0(\omega_X^{DB}\otimes L\otimes \mc O_Z)$$ is surjective.
\end{theorem}
\begin{corollary}\label{cor: abounddb}
Let $X$ be a projective normal complex surface, $Z \subset X$ a subscheme of length $l_Z$, and $L$ an ample line bundle on $X$. Let $l_X := \Length(\mc H^1(\O0))$. Then
$$H^1(X, \omega_X^{DB} \dt L^{\otimes a} \otimes I_Z) = 0$$
for any integer
$ a \ge m'(C_X, l_Z + l_X)$, where $m'$ is defined in Definition \ref{def: m'}. In particular, this implies $\omega_X^{DB} \otimes L^{\otimes a}$ separates jets along $Z,$ i.e. 
$$H^0(\omega_X^{DB} \otimes L^{\otimes a})\rightarrow H^0(\omega_X^{DB} \otimes L^{\otimes a}\otimes \mc O_Z)$$ is surjective.
\end{corollary}

We now briefly explain the strategy of the proof. First, we note the separation of jets along $Z$ of $\omega_X \otimes L$ is implied by the vanishing $H^1(X, \omega_X \dt L \otimes I_Z) = 0,$ which is equivalent to the vanishing of $\Hom(L \otimes I_Z, \ms O_X[1])$. In the case of $\omega_X^{DB}$, the object $\ms O_X[1]$ is replaced by $\O0[1]$, and so we consider more generally the spaces of morphisms from objects of the form $L \otimes I_Z$ to a class of objects ``of type $O$" (see Definition \ref{def: type O} for details), which includes both $\ms O_X[1]$ and $\O0[1]$. As in Arcara-Bertram \cite{arcara2009reider}, the goal is to find Bridgeland stability conditions with respect to which objects of type $O$ (respectively, objects of the form $L \otimes I_Z$) are stable, and then use Schur's lemma to conclude. Refining the argument slightly, we see that in fact the image of a nonzero morphism $f$ of this form (appropriately defined) is a torsion sheaf whose support gives an effective divisor satisfying Reider-type conditions.\newline\newline
\textbf{Outline of the paper.} In Section \ref{sec:preliminaries}, we fix notation and recall the necessary background on Bridgeland stability conditions on a (potentially singular) normal surface. We also briefly introduce Du Bois complexes and discuss some of their properties. In section \ref{sec: Bridgeland} we find conditions guaranteeing that objects of the form $L \otimes I_Z$ and objects of type $O$ are Bridgeland stable, adapting the proof of Arcara and Bertram \cite{arcara2009reider} to the more general setting of normal surfaces. We also explain why these results are stronger than we need for the proof of the main theorem, where it suffices to consider only destabilizing objects of a certain form. In section \ref{sec: theorem}, we deduce the final form of our main technical theorem (Theorem \ref{thm: generalvanishing}), from which all our results follow. \newline\newline
\textbf{Acknowledgements.} We are grateful to Mihnea Popa for suggesting this problem and for insightful conversations throughout the project. We would also like to thank Davesh Maulik, Mircea Musta{\c{t}}{\u{a}, Sung Gi Park, and Wanchun Shen for valuable discussions. 

\section{Preliminaries}\label{sec:preliminaries}
In this section, we review Langer's construction of Bridgeland stability conditions on normal proper surfaces \cite{langer2024bridgeland} (cf. \cite{nuer2023bogomolov}). We also give a brief introduction to Du Bois complexes. In particular, we focus on the zeroth Du Bois complex of a normal surface.

\subsection{Bridgeland stability on normal surfaces}\label{sec:Mumford}
Let $X$ be a normal surface. Langer constructs Chern character homomorphisms
$$\Ch^M_i: K_0(X) \rightarrow A_{2-i}(X)_{\QQ},$$
where $\Ch^M_0$ and $\Ch^M_1$ are defined as usual on the smooth locus and $\Ch^M_2$ is defined so that a Riemann–Roch formula holds (see \cite{langer2022intersection} for details). We will simply denote the homomorphism by $\Ch_i$ in what follows. (This agrees with the standard Chern character on vector bundles, justifying the use of the notation.)

In addition to the Chern character, Langer's definition uses the Mumford intersection product on $A_1(X)$ defined as follows \cite[Ex. 7.1.16]{fulton2013intersection}:
Let $\pi: \tilde{X} \rightarrow X$ be a minimal resolution of singularities. Then we define a homomorphism
$A_1(X) \rightarrow A_1(\tilde{X})_{\QQ}$ by
$$\alpha = [C] \mapsto \alpha' := [\tilde{C}] + \Delta,$$
where $C$ is assumed to be an irreducible curve on $X$, $\tilde{C}$ is the proper transform of $C$, and $\Delta$ is the unique $\mathbb{Q}$-divisor supported on the exceptional locus of $\pi$ such that for any component $E$ of the exceptional locus, we have $\alpha' \cdot E = 0$. For future use, we record the following observations:

\begin{observations} \label{obs:int}
\begin{enumerate}
\item Writing $\sing(X) = \{p_1, \ldots, p_a\}$ (as a normal surface has isolated singularities) and $\pi^{-1}(p_i) = \cup E_{ij}$ the set of irreducible components, then for $N_i := \det((E_{ij} \cdot E_{ik})_{j,k})$ and $N := \lcm_i N_i$, we have $\im(A_1(X)) \subset \frac{1}{N} \ZZ \otimes A_1(\tilde{X})$.
\item If $\alpha$ is the class of a Cartier divisor $D$, then $\alpha' = [\pi^*D] \in A_1(\tilde{X})_\ZZ$.
\item Assuming from now on that $X$ is in addition proper, then the composition of the map $A_1(X) \rightarrow A_1(\tilde{X})_\QQ$ with the standard intersection product $A_1(\tilde{X})_\QQ \otimes A_1(\tilde{X})_\QQ \rightarrow \QQ$ yields an intersection product on $A_1(X)$ satisfying the Hodge index theorem.
% Because if \alpha = [L] we have \alpha' \cdot C' = \alpha' \cdot \tilde{C} = [\pi^* L] \cdot \tilde{C} = [L] \cdot C by the projection formula, and L|_C is ample
\item Given $\alpha,\beta \in A_1(X)$ with $\alpha' \in \frac{1}{M} \ZZ \otimes A_1(\tilde{X})$, then $\alpha \cdot \beta = \alpha' \cdot \tilde{\beta} \in \frac{1}{M} \ZZ$, as $\beta' = \tilde{\beta} + \Delta$ with $\alpha' \cdot \Delta = 0$. In particular, if $\alpha$ is the class of a Cartier divisor, then $\alpha \cdot \beta \in \ZZ$. (More precisely, $\alpha \cdot \beta$ is given by composition of the degree map with the first Chern class intersection map of \cite[Section 2.5]{fulton2013intersection}).
% Because if \alpha = c_1(L) we have \alpha \cdot B = deg \pi_*(c_1(\pi^*L) \cdot B') = deg \pi_*(c_1(\pi^*L) \cdot \tilde{B}) = deg c_1(L) \cdot B
\end{enumerate}
\end{observations}

Now, let $X$ be a normal proper surface over an algebraically closed field, and choose $\RR$-divisors $\omega, B \in A_1(X)_\RR$ such that $\omega$ is numerically ample. Langer defines a Bridgeland stability condition $(Z_{\omega,B}, \Coh_{\omega,B})$ as follows: First, he shows that one can choose a constant $C_X \in \RR_{\ge 0}$ satisfying the Bogomolov-type inequality
$$\int_X \Ch_1(E)^2 - 2 \Ch_0(E) \cdot \Ch_2(E) + C_X \Ch_0(E)^2 \ge 0$$
for any torsion-free, $\omega$-slope semistable coherent sheaf $E$ on $X$. (For example, if $X$ is smooth and defined over an algebraically closed field of characteristic 0, then by the ordinary Bogomolov inequality we may choose $C_X = 0$.
% and in what follows, we will always assume $C_X$ has been chosen to be nonnegative.
For more discussion, see \ref{sec: cx}) 

Given the Bogomolov inequality on $X$, one defines the central charge
$$Z_{\omega,B}: K_0(X) \rightarrow \CC, E \mapsto - \int_X e^{-(B +i\omega)} \Ch(E) + \frac{C_X}{2} \Ch_0(E).$$
We note that this matches the standard definition of the central charge in the smooth case, assuming one takes $C_X = 0$. The heart $\Coh_{\omega,B}$ is defined exactly as in the smooth case, i.e., we define $\Coh_{\omega,B} := \langle \mc F_{\omega,B}, \mc T_{\omega,B} \rangle$ to be the tilt of $\Coh(X)$ by the torsion pair
$$\mc T_{\omega,B} := \{E \in \Coh(X): \textrm{all $\omega$-slope HN factors are of slope $\mu_\omega > B \cdot \omega$}\}$$
$$\mc F_{\omega,B} := \{E \in \Coh(X): \textrm{all $\omega$-slope HN factors are of slope $\mu_\omega \le B \cdot \omega$}\},$$
where the $\omega$-slope of a coherent sheaf $E$ is defined by
$$\mu_\omega(E) := \frac{\Ch_1(E) \cdot \omega}{\Ch_0(E)} \in (-\infty, \infty]$$
(setting $\mu_\omega(E) = \infty$ when $\Ch_0(E) = 0$).

Recall that the central charge encodes the Bridgeland rank
$r_{\omega, B}(E) := \Im Z_{\omega, B}(E)$ and degree $d_{\omega, B}(E) := - \Re Z_{\omega,B}(E)$. For $E \ne 0 \in \Coh_{\omega, B}$ we have $r_{\omega,B}(E) \ge 0$, and in the case of equality $d_{\omega,B}(E) > 0$. We then define the Bridgeland slope
$$\nu_{\omega, B}(E) := \frac{d_{\omega, B}(E)}{r_{\omega, B}(E)} \in (-\infty, \infty].$$
An object of $\Coh_{\omega,B}$ is said to be Bridgeland stable if it has Bridgeland slope greater than that of any proper subobject. Given Bridgeland stable objects $F \ne G \in \Coh_{\omega,B}$ with $\nu_{\omega,B}(F) \ge \nu_{\omega, B}(G)$ we have the important property that $\Hom(F,G) = 0$.

In this paper, we only consider the half-plane of stability conditions $(\omega, B) = (tH, sH)$ for a fixed ample $H \in A_1(X)$ and varying $(t,s) \in \RR^{>0} \times \RR$. We will use $Z_{s,t}$ to denote $Z_{tH, sH}$, and similarly for $r_{s,t}, \mu_{s,t},$ etc. We note that $\Coh_{s,t} =: \Coh_s$ is in fact independent of the choice of $t \in \RR^{>0}$.
For convenience, we record the following:
$$r_{s,t}(E) =  tH \cdot (\Ch_1(E) - s \Ch_0(E) H)$$
$$d_{s,t}(E)= \Ch_2(E) - s \Ch_1(E) \cdot H + \Ch_0(E) \frac{(s^2-t^2) H^2 - C_X}{2}$$

\subsection{More about the constant $C_X$}\label{sec: cx}
Since the constant $C_X$ appears in our Reider-type bounds, we will briefly recall its definition and compute it in an example.

 First, it is not hard to reduce to showing the inequality for an arbitrary reflexive sheaf $E.$ Let $f: \tilde{X} \rightarrow X$ be a minimal resolution of $X$, and set $F := (f^* E)^{**}$. 
 Then one uses the fact that a Bogomolov inequality is satisfied for $F$, i.e.,
$$\int_{\tilde{X}} \Delta(F) \ge 
{-C_{\tilde{X}} r^2}$$ which is classical in the characteristic 0 (with $C_{\tilde{X}}=0$) and a theorem of Koseki in positive characteristic \cite{koseki2023bogomolov}.

Let $c_1(f,F)$ be the divisor supported on the exceptional locus of $f$ uniquely defined by the property that $c_1(f,F) \cdot E_i = F \cdot E_i$ for each irreducible exceptional divisor $E_i$.
Langer calculates in \cite{langer2024bridgeland} that
$$\int_{\tilde{X}} \Delta(F) - \int_{X} \Delta(E) = 2r h^0(X, R^1f_* F) - 2r^2 h^0(X, R^1 f_* \ms O_{\tilde{X}}) + c_1(f,F)^2 - r c_1(f, F) \cdot K_{\tilde{X}} $$
and so we get
$$\int_X \Delta(E) \ge -2r h^0(X, R^1f_* F) + 2r^2 h^0(X, R^1 f_* \ms O_{\tilde{X}}) - c_1(f,F)^2 + r c_1(f, F) \cdot K_{\tilde{X}}{-C_{\tilde{X}}r^2}.$$
In \cite{langer2022intersection}, Langer shows that
%there exists a constant $C_X$ such that $\int_X \Delta(E) \ge - C_X r^2$. 
%The idea is to show that
$$h^0(X, R^1 f_* F) \le (r+2) h^0(X, R^1 f_* \ms O_{\tilde{X}})$$
and so
$$-2rh^0(X, R^1 f_* F) + 2r^2 h^0(X, R^1 f_* \ms O_S)\ge %(-2r(r+2) + 2r^2) h^0(X, R^1 f_* \ms O_S)=
-4r h^0(X, R^1 f_* \ms O_X).$$
Moreover, he calculates that for $E_j$ a component of the exceptional divisor we have
$$0 \le c_1(f, F) \cdot E_j \le (r+2) h^0(X, R^1 f_* \ms O_{\tilde{X}}) - r \chi(\ms O_{E_j}) - r E_j^2.$$
Putting these together, one sees that we can bound $\int_X \Delta(E) \ge - C_X r^2$ for some $C_X$ independent of $E$ (but depending on $C_{\tilde{X}}, h^0(X, R^1f_* \ms O_{\tilde{X}}))$, and the genera and intersection numbers of the $E_j$'s).
\begin{example}
Let $C$ be a degree 3 complex projective plane curve with line bundle $L=\mc O_{\mathbb{P}^2}(1)|_C.$ Then consider $X$ the projectivization of the cone over $C$ with conormal bundle $L$, i.e. of the affine cone
$$\text{Spec}\left(\bigoplus_{m\geq 0} H^0(C,L^m)\right).$$

Note that by \cite[Appendix A]{shen2023k} the cone has a Du Bois but not rational singularity. We would like to determine $C_X$ in this case.

First, note that $h^0(R^1 f_* \mc O_{\tilde{X}})= h^1(\mc O_C)=g(C)=1.$
Moreover, $E \cdot K_{\tilde{X}}=3$ by adjunction. Now, suppose $c_1(f,F)=\alpha E.$ Then, by the above discussion, {and since $C_{\tilde{X}} = 0$ by classical Bogomolov}, one can choose $C_X$ to be the smallest positive number such that $$-4r+3\alpha^2+3r\alpha \geq -C_Xr^2.$$ 
Then $-4r+3\alpha^2+3r\alpha$ is minimized at $\alpha=-\frac{r}{2}$ so we would like to find $C_X$ such that $-4r-\frac34r^2+C_Xr^2\geq 0$ for any integer $r\geq 1.$ This means $C_X=\frac{19}{4}.$

 \end{example}

 \begin{example}
     One can similarly compute a bound for $C_X$ for $X$ the cone over a smooth projective degree $d$ complex plane curve with conormal bundle $L=O_{\mathbb{P}^2}(1)|_C.$ This bound grows asymptotically on the order of $d^3.$  
 \end{example}

\subsection{Du Bois complexes}\label{sec:DB}

Let $X$ be a complex variety. One would like to consider an analogue of the standard de Rham complex on smooth varieties for singular varieties. Fix a hyperresolution $\varepsilon_\bullet: X_\bullet \rightarrow X.$ In \cite{du1981complexe}, following ideas of Deligne, Du Bois introduced $\underline{\Omega}_X^\bullet:= \mathbf{R}\varepsilon_{\bullet *}\Omega^k_{X_\bullet},$ which is an object in the derived category of filtered differential complexes on $X$, and showed that this is independent of the choice of hyperresolution. One can associate a filtration $ F^p\underline{\Omega}_X^\bullet:=\mathbf{R} \varepsilon_{\bullet *} \Omega_{X_\bullet}^{\geq p}$ by recalling that $\Omega_{X_i}^\bullet$is filtered by $\Omega_{X_i}^{\geq p}.$ Consider then the $p$-th Du Bois complex of $X$ $$\underline{\Omega}_X^p:=\text{gr}_F^p\underline{\Omega}_X^\bullet[p].$$ For more details on Du Bois complexes, we refer the reader to \cite{du1981complexe}, \cite[Chapter V]{guillen2006hyperresolutions}, or \cite[Chapter 7.3]{peters2008mixed}.

In this paper, we will be concerned with understanding the stability of $\underline{\Omega}_X^0,$ 
the 0-th Du Bois complex of $X.$ A result of Saito and Schwede \cite[Proposition 5.2]{saito2000mixed}, \cite[Lemma 5.6]{schwede2009f}  says that for any complex variety $X,$ $$\mc H^0 (\O0)=\ms O_{X^{\text{sn}}},$$ where $\ms O_{X^{\text{sn}}}$ is the structure sheaf of the seminormalization of $X.$

In our case, since $X$ is normal, note that we simply have $\mc H^0 (\O0)=\ms O_X.$

Moreover, for each $p \geq 0$, there is a canonical morphism $\Omega^k_p\rightarrow \underline{\Omega}_X^p$, which is an isomorphism if $X$ is smooth; see \cite[Page 175]{peters2008mixed}. In particular, the $\mc H^i(\underline{\Omega}_X^p)$ are supported on the singular locus of $X,$
for all $i> 0.$ Hence, in our case, $\mc H^1 (\O0)$ is supported in codimension 2. General vanishing results
for the cohomologies of Du Bois complexes show that $\mc H^i(\underline{\Omega}_X^0)=0$ for $i\neq 0,1.$
 
Note that any normal surface $X$ is Cohen-Macaulay (since it is $S_2$). Let $\omega_X$ be the dualizing sheaf of $X$, and consider the duality functor $\DD:= \RHom(-, \omega_X)$ on $D^b(X).$ We define $$\omega_X^{DB}:= \RHom(\O0, \omega_X).$$ 
Now, by the injectivity theorem of Kov{\'a}cs-Schwede 
\cite[Theorem 3.3]{kovacs2016dubois},
we have that the map
$$\RHom(\O0, \omega_X)\rightarrow \RHom( \mc O_X, \omega_X ),$$ induced by dualizing the canonical morphism $\mc O_X\rightarrow \O0,$ is injective on cohomology, i.e., $\sExt^i(\O0, \omega_X) \hookrightarrow \sExt^i(\ms O_X, \omega_X)$ for all $i$. In other words, $\omega_X^{DB}$ is a subsheaf of $\omega_X$.
%since we have an injection $\omega_X^{DB}=\RHom(\O0, \omega_X) \xhookrightarrow{} \omega_X$.

We conclude this section with a result that will be necessary in order to apply our main theorem.
\begin{proposition}\label{prop: omegaDB sheaf}
For any closed point $p \in X$ we have $\Hom_{D^b(X)}(\ms O_p, \O0[1]) = 0$.
\end{proposition}

\begin{proof}

We note first that $\RHom(\ms O_p, \omega_X) = \ms O_p[-2]$, as $\RHom(\ms O_p, \omega_X)$ is supported on $p$ and so by the local-to-global spectral sequence it suffices to compute
$\Ext^i(\ms O_p, \omega_X) = H^{2-i}(X, \ms O_p)^*$. Then
$$\Hom(\ms O_p, \O0[1]) = \Hom(\ms O_p, \RHom(\DD(\O0), \omega_X)[1]) = \Hom(\ms O_p \dt \DD(\O0), \omega_X[1]) $$
$$= \Hom(\DD(\O0), \RHom(\ms O_p, \omega_X)[1]) = \Hom(\DD(\O0), \ms O_p[-1]) = 0$$
% By Stacks 36.11.2 D^b(Coh(O_X)) is a full subcategory of D(O_X), and this is the definition of RHom in 20.42.0.1
by vanishing of negative Exts since $\DD(\O0)$ is a sheaf, as discussed above.
%(see \cite[Lemma 1.4]{kovacs1999rational}, \cite[Theorem 5.1]{kovacs2011hodge}).
\end{proof}

We note that this proposition cannot be deduced purely from the cohomology sheaves $\mc H^i(\O0)$ but requires the additional input that $\DD(\O0)$ is a sheaf; indeed, the statement would be false in the non-Du Bois case if $\O0$ were replaced by $\mc H^0(\O0) \oplus \mc H^1(\O0)[-1]$.

\section{Bridgeland stability}\label{sec: Bridgeland}
In this section we prove several lemmas about Bridgeland stability of elements of $D^b(X)$ whose cohomologies match those of the objects we will consider. Namely,  we study the stability of a class of objects with properties matching those of $\ms O_X[1]$ and $\O0[1]$ (``type $O$") and of $L \otimes \mc I_Z$ for $L$ an ample line bundle and $Z \subset X$ a subscheme of dimension 0. We also analyze the image of a morphism from an object $L \otimes I_Z$ to an object of type $O$.

\subsection{Conditions for stability}
In this section, we derive inequalities that would be satisfied if an object of one of the two forms of interest were not Bridgeland stable, following the proof of \cite{arcara2009reider} with the necessary modifications for the non-smooth case.
\subsubsection{Objects of type $O$} 

The following definition captures the essential properties of $\ms O_X[1]$ and $\O0[1]$.
\begin{definition}\label{def: type O}
We say an object $F \in D^b(X)$ is \emph{of type} $O$ if $\Hom(\ms O_p, F) = 0$ for every closed point $p \in X$ and
$$\mc H^i(F) = \begin{cases} \ms O_X & i = -1 \\ \ms O_T, \text{ for } T \subset X \textrm{ a zero-dimensional subscheme} & i = 0 \\ 0 & \textrm{else}.
\end{cases}$$
\end{definition}

\begin{example}
As explained in section \ref{sec:DB}, the object $\O0[1]$ is of type $O$. Because $\RHom(\ms O_X, \omega_X) = \omega_X$ is also a sheaf,
the argument of Proposition \ref{prop: omegaDB sheaf} shows that $\Hom(\ms O_p, \ms O_X[1]) = 0$ for any $p \in X$, so $\ms O_X[1]$ is also of type $O$. Another example is the shifted derived dual $\ms I_Z^\vee[1]$, where $Z \subset X$ is a 0-dimensional subscheme contained in the Gorenstein locus of $X$; this is the type of object considered (in the smooth case) in \cite{arcara2009reider}.
\end{example}

Fix some numerically ample $H \in A_1(X)_\RR$, and let $F \in D^b(X)$ be an element of type $O$.
Our goal will be to find a choice of $(t,s)$ such that $F$ is a Bridgeland stable element of $\Coh_{s,t}$.
\begin{proposition}
An object $F$ of type $O$ belongs to the heart $\Coh_{s,t}$ if and only if $s \ge 0$.
\end{proposition}
\begin{proof}
We note that $\mc H^0(F) = \ms O_T$ is torsion, therefore is contained in $\mc T_{s,t}$ for any choice of $s$ and $t$. So it suffices to see that $\mc H^{-1}(F) = \ms O_X \in \mc F_{s,t}$. Since $\ms O_X$ is of rank 1, and thus is automatically $H$-stable, the necessary condition is that
$$tH \cdot sH = stH^2 \ge \mu_H(\ms O_X) = \frac{\Ch_1(\ms O_X) \cdot tH}{\Ch_0(\ms O_X)} = 0,$$
and since necessarily $t > 0$, this amounts to the condition $s \ge 0$.
\end{proof}
We now explore the implications of $F$ being Bridgeland unstable.

\begin{proposition}\label{prop:rk ineq}
Fix $(t,s) \in \RR^{>0} \times \RR^{\ge 0}$. If $F \in \Coh_{s,t}$ is destabilized by some quotient $Q^\bullet$, then for any $H$-slope HN factor $E$ of $\mc H^{-1}(Q^\bullet)$ with Bridgeland rank $r_{s,t}(E) \ne -r_{s,t}(F)$, we have that
$$(r-1) sH^2 < A \cdot H \le rsH^2,$$
where we write $r := c_0(E)$ and $A := \Ch_1(E)$.
\end{proposition}

\begin{proof}
Suppose we have a destabilizing short exact sequence
$$0 \rightarrow K^\bullet \rightarrow F \rightarrow Q^\bullet \rightarrow 0$$
in $\Coh_{s,t}$. Then since Bridgeland rank is additive in short exact sequences and nonnegative on $\Coh_{s,t}$, we have that
$$0 \le r_{s,t}(E[1]) \le r_{s,t}(Q^\bullet) \le r_{s,t}(F)$$
for $E$ as described above, as $Q^\bullet$ is a subobject of $F$ in $\Coh_{s,t}$ and can be built up from a series of extensions of $\mc H^0(Q^\bullet)$ and $\{E_i[1]\}$, where the $E_i$'s run over the $H$-slope HN factors of $\mc H^{-1}(Q^\bullet)$. Thus, using our assumption, we get
$$0 \le r_{s,t}(E[1]) = - tH \cdot (A - srH) < r_{s,t}(F) = stH^2 \implies 
srH^2 \ge H \cdot A > s(r-1) H^2.$$
\end{proof}

Define $l_T := \Length(\mc H^0(F))$. We start by noting that $\Ch(F) = \Ch(\ms O_T) - \Ch(\ms O_X)$, where $\ms O_X$ is a line bundle and therefore Langer's Chern character is the standard one. On the other hand, we have $\Ch_0(\ms O_T)=\Ch_1(\ms O_T) = 0$ and so $\int_X \Ch_2(\ms O_T) = \chi(X, \ms O_T) = l_T$ by \cite{langer2022intersection}. 

In the following proposition, we will restrict attention to one particular stability condition, chosen so that $s = \frac{1}{2}$ and $\nu_{s,t}(F) = 0$ (taking $l = l_T$).
% because $\nu_{s,t}(F) = \frac{l_T - \frac{(s^2-t^2) H^2 - C_X}{2}}{stH^2}$ so the point is $2l_T - (s^2-t^2) H^2 + C_X = 0$, i.e., $s^2 - t^2 = (2l_T + C_X)/H^2$

\begin{proposition}\label{prop:deg ineq}
Suppose $H^2 > 4(2l + C_X)$, and fix $s := \frac{1}{2}, t := \sqrt{\frac{1}{4} - \frac{2l + C_X}{H^2}}$. Suppose $E$ is an $H$-slope semistable sheaf with $r_{s,t}(E) \le 0$ such that $\nu_{s,t}(E) \le 0$. Then
$$A \cdot H \le \frac{A^2}{r} + r(C_X + 2l),$$
defining $A := \Ch_1(E)$ and $r := \Ch_0(E)$ as before.
\end{proposition}

\begin{proof}
The slope inequality $\nu_{s,t}(E) \le 0$ with denominator $r_{s,t}(E) \le 0$ implies that
the numerator
$$d_{s,t}(E) = \Ch_2(E) - sA \cdot H + r \frac{(s^2-t^2)H^2 - C_X}{2} \ge 0$$
$$\implies A \cdot H \le 2\Ch_2(E) + 2rl.$$
%2 \Ch_2(E) - A \cdot H + r(2l - C_X+C_X) \ge 0
Now, by Langer's Bogomolov-type inequality \cite[Theorem 0.1]{langer2024bridgeland}, we have
$$2r\Ch_2(E) \le A^2 + C_Xr^2 \implies A \cdot H \le \frac{A^2}{r} + C_X r + 2rl.$$
% \Ch_1^2 - 2\Ch_0 \Ch_2 + C_X \Ch_0^2 \ge 0 --> 2\Ch_0 \Ch_2 \le \Ch_1^2 + C_X \Ch_0^2
\end{proof}

We now combine the previous two propositions to deduce a contradiction if $H$ is sufficiently positive.

\begin{lemma}\label{lemma:Hodge}
Let $H=c_1(L),$ where $L$ is an ample line bundle such that
 $H^2 > (C_X + 2l+1)^2$. 
Given $A \in A_1(X), r \in \ZZ_{\ge 1}$ satisfying
$$ \frac{r-1}{2} H^2 < A \cdot H \le \frac{r}{2} H^2, \,\, A \cdot H \le \frac{A^2}{r} + r(C_X + 2l),$$
we must have either $r = 1$ and $0 < A \cdot H < C_X + 2l+1$, or $r \ge 3$ and $H^2 < 3(C_X + 2l + 1)$ (so in particular $l = 0$ and $C_X < 2$). In the case $r=1$, the same holds when $H$ is any numerically ample $\RR$-divisor.
\end{lemma}

\begin{proof}
Define $D := A/r$. We start by showing that in the case $r > 1$ we must have $D^2 \ge 1$. Indeed, suppose first that $r\geq 3.$
Then by our assumptions we have 
$$D^2 + C_X + 2l \ge D \cdot H > \frac{r-1}{2r} H^2 \geq \frac{H^2}{3} > C_X + 2l+1$$ 
unless $H^2 \le 3(C_X + 2l + 1)$.
When $r=2$, we also use the assumption that $H$ is (integral) Cartier: recalling from \ref{obs:int}(4) that $H^2$ and $A\cdot H$ are integers, we have $$D \cdot H >  \frac{H^2}{4}> \frac{4(C_X+2l)+1}{4}$$ which implies $\frac{H^2}{4}\geq C_X+2l+\frac 12$ and $D \cdot H \geq C_X+2l +1,$ and so as above $D^2 \ge 1$.

We now show that our assumption on $H^2$ implies that $D^2 < 1$:
recalling from \ref{obs:int}(3) that the intersection product on $X$ satisfies the Hodge index theorem, we have
$$D^2 H^2 \le (D \cdot H)^2 \le (D^2 + C_X + 2l) \frac{H^2}{2} \implies D^2 \le C_X + 2l$$
and
$$D^2 H^2 \le (D \cdot H)^2 \le (D^2 + C_X + 2l)^2 \implies H^2 \le \left(\sqrt{D^2} + \frac{C_X + 2l}{\sqrt{D^2}}\right)^2.$$
If we had $D^2 \ge 1$, using that $D^2 \le C_X + 2l$ we would conclude that
$$H^2 \le (C_X + 2l+1)^2.$$
However, this contradicts our assumption on $H$, so we conclude that in fact $D^2 < 1$, which implies by the above that $r=1.$

It follows that
$$A \cdot H \le A^2 + C_X + 2l < C_X + 2l+1.$$
\end{proof}

We now describe conditions under which an object $F$ is Bridgeland stable.

\begin{proposition}\label{prop: O}
Let $F \ne O_X[1] \in D^b(X)$ be of type $O$, and suppose $H=c_1(L),$ where $L$ is an ample line bundle such that
$H^2 > (C_X + 2l_T+1)^2$ 
and $H \cdot A_1(X) \subset M \ZZ$ for some $M \ge C_X + 2l_T+1$.
% note that (y+1)^2-4y = (y-1)^2 \ge 0 so in particular (C_X + 2l_T + 1)^2 \ge 4(C_X + 2l_T) always
Then $F$ is Bridgeland stable for the stability condition $(Z_{s,t}, \Coh_{s,t})$ with $s = \frac{1}{2}, t = \sqrt{\frac{1}{4} - \frac{2l_T + C_X}{H^2}}$. (If $F = \ms O_X[1]$, the same holds with the additional bound $H^2 \ge 3(C_X + 1)$.)
\end{proposition}

\begin{proof}
Suppose $F \in \Coh_{s,t}$ is destabilized by a quotient $Q^\bullet$, i.e., $0 = \nu_{s,t}(F) \ge \nu_{s,t}(Q^\bullet)$. Our strategy will be to find an $H$-slope semistable factor $E$ of $\mc H^{-1}(Q^\bullet)$ with $r_{s,t}(E) \ne -r_{s,t}(F)$ and $\nu_{s,t}(E) \le \nu_{s,t}(F)$ and then to apply Lemma \ref{lemma:Hodge}.

We start by showing that our assumption that $\Hom(\ms O_p, F)$ for every closed point $p$ implies $r_{s,t}(Q^\bullet) < r_{s,t}(F)$ strictly: Suppose $r_{s,t}(Q^\bullet) = r_{s,t}(F)$, and consider the kernel object $K^\bullet$. We have the long exact sequence of cohomology
$$0 \rightarrow \mc H^{-1}(K^\bullet) \rightarrow \ms O_X \rightarrow \mc H^{-1}(Q^\bullet) \rightarrow \mc H^0(K^\bullet) \rightarrow \ms O_T \rightarrow \mc H^0(Q^\bullet) \rightarrow 0,$$
where all the cohomologies $\mc H^{-1}$ are by definition torsion free. This implies immediately that $\mc H^{-1}(K^\bullet) \cong 0$ or $\ms O_X$, and the latter is impossible as in this case
$$r_{s,t}(K^\bullet) = 0 = r_{s,t}(\mc H^0(K^\bullet)) - r_{s,t}(\mc H^{-1}(K^\bullet))$$
where $r_{s,t}(\mc H^0(K^\bullet)) \ge 0$ and $r_{s,t}(\ms O_X) = -stH^2 < 0$. So we conclude $\mc H^{-1}(K^\bullet) = 0$, and then $r_{s,t}(\mc H^0(K^\bullet)) = 0$ implies, by definition of $\Coh_{s,t}$, that $\mc H^0(K^\bullet)$ is supported in dimension 0 (as this can be decomposed into a torsion sheaf and torsion-free sheaves with $r_{s,t} > 0$, and any torsion components supported in dimension 1 would pair positively with the ample divisor $H$, again giving $r_{s,t} > 0$). Now, restricting the inclusion $K^\bullet \hookrightarrow F$ to any length one subsheaf $\ms O_p \subset \mc H^0(K^\bullet)$ gives a nonzero element of $\Hom(\ms O_p, F)$, a contradiction. Thus we have
$$-r_{s,t}(E) = r_{s,t}(E[1]) \le r_{s,t}(\mc H^{-1}(Q^\bullet)) = r_{s,t}(Q^\bullet) < r_{s,t}(F)$$
for any semistable factor $E$ of $\mc H^{-1}(Q^\bullet)$.

Now, recall that by definition of $\Coh_{s,t}$ and the long exact sequence of cohomology, we have that $\mc H^0(Q^\bullet)$ is a quotient of $\mc H^0(F) = \ms O_T$, and $\mc H^{-1}(Q^\bullet)$ has a filtration whose quotients are torsion-free $H$-slope semistable sheaves $E_i$ with $r_{s,t}(E_i) \le 0$. Then $\mc H^0(Q^\bullet)$, being supported in dimension 0, has Bridgeland rank 0 and nonnegative degree, and so $\nu_{s,t}(\mc H^{-1}(Q^\bullet)) \le \nu_{s,t}(Q^\bullet)$. Now, $\nu_{s,t}(\mc H^{-1}(Q^\bullet))$ is a weighted average of the $\nu_{s,t}(E_i)$, and so at least one of these, say, $E$, must have Bridgeland slope $\nu_{s,t}(E) \le \nu_{s,t}(\mc H^{-1}(Q^\bullet)) \le \nu_{s,t}(F) = 0$, and by the previous paragraph we have $r_{s,t}(E) \ne -r_{s,t}(F)$.

Finally, we can apply Propositions \ref{prop:rk ineq} and \ref{prop:deg ineq} to conclude that $r := c_0(E)$ and $A := \Ch_1(E)$ satisfy
$$\frac{r-1}{2} H^2 < A \cdot H \le \frac{r}{2} H^2, \,\, A \cdot H \le \frac{A^2}{r} + r(C_X + 2l_T).$$
Then by Lemma \ref{lemma:Hodge} we get $A \cdot H \le C_X + 2l_T+1$, but this is impossible by our assumption that $H \cdot A_1(X) \subset M \ZZ$ with $M > C_X + 2l_T+1$.
\end{proof}

In fact, to obtain the optimal bounds in our main theorem, we will use the following variant of the above proposition with a slightly weaker assumption on $H$ (which can now be an $\RR$-divisor, and the inequality in the bounds is no longer strict). This small change turns out to be necessary to obtain the optimal Fujita-type bounds, at least in the smooth case.

\begin{proposition}\label{prop: Ork1}
Let $F \in D^b(X)$ be of type $O$ and let $H$ be an numerically ample $\RR$-divisor such that $H^2 \ge (C_X + 2l_T + 1)^2$ 
and $H \cdot C \ge C_X + 2l_T + 1$ for any nonzero effective divisor $C$. Then $F$ is not strictly destabilized by an injection from a torsion sheaf in cohomological degree 0 with respect to the Bridgeland stability condition $(Z_{s,t}, \Coh_{s,t})$ for $s=\frac{1}{2}, t = \sqrt{\frac{1}{4} - \frac{2l_T + C_X}{H^2}}$ if $t > 0$ (i.e., in any case except $l_T = 0, C_X = 1, H^2 = 4$).
\end{proposition}

\begin{proof}
Suppose $K^\bullet \hookrightarrow F$ is a strictly destabilizing injection in $\Coh_{1/2}$, where $K^\bullet = \mc H^0(K^\bullet)=: K$ is torsion. Let $Q^\bullet$ be the quotient object. Then by the long exact sequence of cohomology we conclude that $E := \mc H^{-1}(Q^\bullet)$ is of rank 1, so in particular this is $H$-slope semistable, and since $Q^\bullet$ is a destabilizing quotient, we must have $r_{s,t}(E) \ne 0$ (or else, since $\mc H^0(Q^\bullet)$ is supported in dimension 0, we would conclude that $\nu_{s,t}(Q^\bullet) = \infty > \nu_{s,t}(F)$).

For any $\epsilon > 0$ define $H_\epsilon :=(1+\epsilon)H$, $s_\epsilon := \frac{1}{2}$, and $t_\epsilon := \sqrt{\frac{1}{4} - \frac{2l_T + C_X}{H_\epsilon^2}}$. We note that $K^\bullet \in \Coh_{s_\epsilon H_\epsilon}$ for all $\epsilon$, and $Q^\bullet \in \Coh_{s_\epsilon H_\epsilon}$ for $\epsilon$ sufficiently small (where here we use that $\mu_{tH}(E) < stH^2$ with strict inequality because $r_{s,t}(\mc H^{-1}(Q^\bullet)) \ne 0$ and thus the same holds for $\epsilon$ sufficiently small). Thus the exact triangle $K^\bullet \to F \to Q^\bullet \xrightarrow{+1}$ is contained in $\Coh_{s_\epsilon H_\epsilon}$ and so $F \to Q^\bullet$ continues to be a surjection in $\Coh_{s_\epsilon}$ for $\epsilon$ sufficiently small. Furthermore, since Bridgeland slope is a continuous function of $\epsilon$ (since the denominators $r_{s,t}(F), r_{s,t}(Q^\bullet) \ne 0$), we conclude that $Q^\bullet$ must be a strictly destabilizing quotient of $F$ for $\epsilon$ sufficiently small.

Now, we apply the argument of the proof of Proposition \ref{prop: O}: first, by definition of $F$ being of type $O$ we know that $K$ must be supported in dimension 1 (as this is torsion but not supported in dimension 0), and so in particular has positive Bridgeland rank. Also $\mc H^0(Q^\bullet)$ is supported in dimension 0, so we have
$$0 \le r_{s_\epsilon, t_\epsilon}(Q^\bullet) = -r_{s_\epsilon, t_\epsilon}(E) < r_{s_\epsilon, t_\epsilon}(F)$$
and also
$$\nu_{s_\epsilon, t_\epsilon}(E) \le \nu_{s_\epsilon, t_\epsilon}(Q^\bullet) \le \nu_{s_\epsilon, t_\epsilon}(F) = 0.$$
Thus we can apply Propositions \ref{prop:rk ineq} and \ref{prop:deg ineq}
% (1+\epsilon)^2 H^2 > 4(2l_T - C_X) --- but this is implied by the later bound
to end up in the situation of Lemma \ref{lemma:Hodge} with $r = 1$.
% (1+\epsilon)^2 H^2 > (C_X + 2l +1)^2
Then we conclude $\Ch_1(E) \cdot H < C_X + 2l + 1$, a contradiction since $\Ch_1(E)$ is the effective divisor given by the support of $K$.
\end{proof}

% Possible remark: by fixing $H'$ and varying $a$ (sufficiently large), we see that F is Bridgeland stable in the large volume limit

\subsubsection{Objects of the form $L \otimes \mc I_Z$}
Let $L$ be an ample line bundle and $Z \subset X$ be a subscheme of dimension 0. We write $l_Z := \Length(Z)$ and fix $H := \Ch_1(L) \in A_1(X)$. Our goal is to show that $L \otimes \mc I_Z$ is Bridgeland stable for the stability condition $s = \frac{1}{2}, t = \sqrt{\frac{1}{4} - \frac{2l_Z + C_X}{H^2}}$ considered in the last section (here with $l_Z$ replacing $l_T$).

We note first of all that $L \otimes \mc I_Z$ is $H$-slope semistable as this is of rank 1, and so $L \otimes \mc I_Z \in \Coh_{s,t}$ exactly when
$$r_{s,t}(L \otimes \mc I_Z) = r_{s,t}(L) = tH \cdot (H - sH) > 0, $$
i.e., when $s < 1$. So in particular, $L \otimes \mc I_Z \in \Coh_{s,t}$ for our choice of $s = \frac{1}{2}$.

\begin{proposition}\label{prop: rk1}
Let $L$, $Z$, $s$, and $t$ be as above, and suppose $H := \Ch_1(L)$ satisfies $H^2 > 4(2l_Z + C_X)$ and $H \cdot C \ge 2(2l_Z + C_X)$ for any nonzero effective divisor $C$.
Then $L \otimes I_Z \in \Coh_{s,t}$ is not destabilized by a subsheaf $L \otimes I_Y, Z \subsetneq Y \subsetneq X$.
\end{proposition}

\begin{proof}
Let $Y_1 \in A_1(X)$ be dimension 1 component of $[Y] \in A_*(X)$. We note immediately that if $Y_1 = 0$, then the sheaf $L \otimes I_Y$ is not destabilizing, as $r_{s,t}(L \otimes I_Y) = r_{s,t}(L \otimes I_Z) \ne 0$ and $d_{s,t}(L \otimes I_Y) = d_{s,t}(L \otimes I_Z) - \Length(I_Z/I_Y)$.

Then since $L \otimes I_Y$ is of rank 1 and thus $H$-slope semistable, we see that $L \otimes I_Y \in \Coh_{s,t}$ if and only if
$$r_{s,t}(L \otimes I_Y) > 0 \iff H^2 > 2H \cdot Y_1 > 0$$
where the second inequality follows from the fact that $Y_1 \ne 0$ is effective.

Now, for $L \otimes I_Y \in \Coh_{s,t}$ to be a destabilizing subsheaf we must have
$$\nu_{s,t}(L \otimes I_Y) \ge \nu_{s,t}(L \otimes I_Z) = 0,$$
and since $r_{s,t}(L \otimes I_Y) > 0$ this means
$$d_{s,t}(L \otimes I_Y) = \Ch_2(L \otimes I_Y) - \frac{1}{2} (H-Y_1) \cdot H + l_Z \ge 0.$$
By Langer's Bogomolov inequality \cite[Theorem 0.1]{langer2024bridgeland} we have
$$2\Ch_2(L \otimes I_Y) \le (H - Y_1)^2 + C_X$$
and so this would imply
$$(H - Y_1)^2 - (H - Y_1) \cdot H + C_X + 2l_Z \ge 0.$$
In other words, to prove our result, it suffices to show
$$ C_X + 2l_Z+ Y_1^2 < Y_1 \cdot H.$$
It follows from the Hodge index theorem (see \ref{obs:int}(3)) and the condition that $L \otimes I_Y \in \Coh_{s,t}$ above that
$$Y_1^2 \le \frac{(Y_1 \cdot H)^2}{H^2} < \frac{Y_1 \cdot H}{2}.$$
Now since $Y_1$ is effective, by assumption we have
$$C_X+2l_Z\leq \frac{Y_1 \cdot H}{2}.$$
Putting the two together $$ C_X + 2l_Z+ Y_1^2 < Y_1 \cdot H,$$ completing the proof.
\end{proof}

\begin{proposition}\label{prop:liz}
Let $L$, $Z$, $s$, $t$ be as above, and suppose $H := \Ch_1(L)$ satisfies
$H^2 >( C_X + 2l_Z + 1)^2$
and $H \cdot A_1(X) \subset M \ZZ$ for some $M \ge \max(2(C_X + 2l_Z), C_X + 1)$.
Then $L \otimes I_Z \in \Coh_{s,t}$ is stable with respect to the Bridgeland stability condition $(Z_{s,t}, \Coh_{s,t})$.
\end{proposition}

\begin{proof}
Suppose that $L \otimes I_Z$ is destabilized by some quotient $Q^\bullet$, i.e., that $0 = \nu_{s,t}(L \otimes I_Z) \ge \nu_{s,t}(Q^\bullet)$. Using the short exact sequence in $\Coh_{s,t}$ given by the exact triangle
$$\mc H^{-1}(Q^\bullet)[1] \rightarrow Q^\bullet \rightarrow \mc H^0(Q^\bullet) \xrightarrow{+1}$$
we conclude that at least one of $\nu_{s,t}(\mc H^{-1}(Q^\bullet)[1]) = \nu_{s,t}(\mc H^{-1}(Q^\bullet))$ and $\nu_{s,t}(\mc H^0(Q^\bullet))$ satisfies this inequality as well.

If $\nu_{s,t}(\mc H^0(Q^\bullet)) \le 0$, then $L \otimes I_Z \twoheadrightarrow Q^\bullet \twoheadrightarrow \mc H^0(Q^\bullet)$ is a destabilizing quotient whose kernel is a subsheaf of $L \otimes I_Z$, but this is impossible by Proposition \ref{prop: rk1}.

On the other hand, if this is not the case, then $\nu_{s,t}(\mc H^{-1}(Q^\bullet)) \le 0$ and $\mc H^{-1}(Q^\bullet) \ne 0$, and we proceed as in the proof of Proposition \ref{prop: O}: we may choose an $H$-slope stable factor $E$ of $\mc H^{-1}(Q^\bullet)$ with $\nu_{s,t}(E) \le \nu_{s,t}(\mc H^{-1}(Q^\bullet)) \le 0$ and $r_{s,t}(E) \le 0$, so that Proposition \ref{prop:deg ineq} applies. Furthermore, since $\mc H^{-1}(Q^\bullet) \ne 0$ and thus $\Ch_0(\mc H^{-1}(Q^\bullet)) > 0$, the same is true of $\mc H^0(K^\bullet)$, letting $K^\bullet = \mc H^0(K^\bullet)$ be the kernel object. Then we have $r_{s,t}(K^\bullet) > 0$ by definition of $\mc T_{s,t}$, and thus the inequality $r_{s,t}(Q^\bullet) < r_{s,t}(L \otimes I_Z)$ is strict. In other words, we get
$$0 \le r_{s,t}(E[1]) = -r_{s,t}(E) < r_{s,t}(L \otimes I_Z) \implies \frac{r-1}{2} H^2 < A \cdot H \le \frac{r}{2} H^2,$$
writing $r := \Ch_0(E)$ and $A := \Ch_1(E)$ as before. Then by Lemma \ref{lemma:Hodge} we conclude that $0 < A \cdot H < C_X + 2l_Z+1$, but by our assumption on $H \cdot A_1(X)$ this is impossible. (We have $C_X + 2l_Z + 1 \le 2(C_X + 2l_Z)$ unless $l_Z = 0, C_X < 1$, in which case the correct bound is $C_X + 1$.)

\end{proof}

\subsection{Reduction to the torsion sheaf case}
As usual, we fix a projective normal surface $X$, an ample line bundle $L$ on $X$, a finite length subscheme $Z$ of length $l_Z$, and an object $F$ of type $O$ with $\Length(\mc H^0(F)) = l_T$. Setting $H := c_1(L)$, recall that $L \otimes I_Z, F$ belong to the heart $\Coh_{1/2}$ defined in Section \ref{sec:preliminaries}. In particular, as this is a full abelian subcategory of $D^b(X)$, it makes sense to talk about the image of a homomorphism $f \in \Hom(L \otimes I_Z, F)$.
\begin{proposition}\label{prop:r=1}
Let $X,L,Z,F, l_Z, l_T$ be as above, and suppose $H^2 > 4(l_Z + l_T + C_X)$.
Then given $f \ne 0 \in \Hom(L \otimes I_Z, F)$ we have that $\im f = \mc H^0(\im f)$ is a torsion sheaf.
\end{proposition}

\begin{proof}
Let $G := \im f$. We use the fact that $L \otimes I_Z \twoheadrightarrow G \hookrightarrow F$, combined with the definition of $\Coh_{1/2}$.

First, since $L \otimes I_Z \twoheadrightarrow G$, the corresponding short exact sequence in $\Coh_{1/2}$ gives a long exact sequence of cohomology
$$0 \to \mc H^{-1}(G) \to \mc H^0(\ker f) \to L \otimes I_Z \to \mc H^0(G) \to 0$$
which implies that $c_0(\mc H^0(G)) \le 1$ with equality if and only if $\mc H^0(G) = L \otimes I_Z$ (as $L \otimes I_Z$ is torsion-free of rank 1). Furthermore, in this case $\mc H^{-1}(G) = \mc H^0(\ker f) = 0$ by definition of $\Coh_{1/2}$, so we conclude that either $c_0(\mc H^0(G))= 0$ or $G = L \otimes I_Z$. We claim the latter is impossible: indeed, since
$$r_{1/2,t}(L \otimes I_Z) = r_{1/2,t}(F) = \frac{t}{2} H^2$$
we conclude that the quotient $\coker f \ne 0$ must satisfy
$$d_{1/2,t}(\coker f) = d_{1/2,t}(F) - d_{1/2,t}(L \otimes I_Z) > 0$$
for all $t>0$. However, our assumption on $H^2$ implies that this is false for $t$ sufficiently small.
% d_{1/2,t}(L \otimes I_Z) = H^2/2 - l_Z - H^2/2 + (1/4 - t^2)H^2/2 - C_X/2
% d_{1/2,t}(F) = l_T - (1/4 - t^2) H^2/2 + C_X/2
% d(F) > d(L \otimes I_Z) iff
% l_T + l_Z - (1/4-t^2) H^2 + C_X > 0
% and as t \to 0 we have 1/4-t^2 \to 1/4
Thus $c_0(\mc H^0(G)) = 0$.

Second, since $G \hookrightarrow F$, we have the long exact sequence
$$0 \to \mc H^{-1}(G) \to \ms O_X \to \mc H^{-1}(\coker f) \to \mc H^0(G) \to \mc H^0(F) \to \mc H^0(\coker f) \to 0.$$
Since $\ms O_X$ and $\ms O_X/\mc H^{-1}(G) \hookrightarrow \mc H^{-1}(\coker f)$ are torsion-free, we conclude that either $\mc H^{-1}(G) = 0$ or $\mc H^{-1}(G) = \ms O_X$. In the latter case, the fact that $\mc H^{-1}(\coker f)$ is torsion-free and injects into $\mc H^0(G)$, which is of rank 0 by the above, implies that $\mc H^{-1}(\coker f) = 0$, so in particular $\mc H^0(G)$ is a subsheaf of $\mc H^0(F)$, therefore of dimension 0. But then we see that
$$\mu_{tH}(\mc H^0(\ker f)) = tH \cdot \frac{\Ch_1(\ms O_X) + \Ch_1(L \otimes I_Z)}{2} = \frac{tH^2}{2} = \frac{H}{2} \cdot tH$$
which is impossible by definition of $\Coh_{1/2}$. So in fact $\mc H^{-1}(G) = 0$.
\end{proof}

%\begin{remark}
%    Using this proposition, one sees that for the purposes of the main theorem, $C_X$ can be replaced by any constant $C_X'$ such that $\int_X \Delta(E) + C_X \ge 0$ for any slope-semistable sheaf of rank $\le 2$ on $E$. A priori such a $C_X'$ might be smaller, as one does not have to ensure this inequality holds for all slope-stable sheaves of higher rank.
%\end{remark}

\section{Proof of the main results}\label{sec: theorem}

We start by translating our separation of jets statement into a form that can be more effectively approached using Bridgeland stability.

\begin{proposition}\label{prop:restatement}
Let $X$ be a projective normal surface, and $G \in D^b(X)$. We have
$$H^1(X, \omega_X^{DB} \dt G) = \Hom(G, \O0[1])^*.$$
\end{proposition}

\begin{proof}
We have
$$H^1(X, \omega_X^{DB} \dt G) = H^1(X, \RHom(\O0, \omega_X) \dt G) = H^1(X, \RHom(\RHom(G, \O0), \omega_X))$$
by \cite[\href{https://stacks.math.columbia.edu/tag/0G4I}{Tag 0G4I}]{stacks-project}, and since $\DD := \RHom(-,\omega_X)$ gives an involution of $D^b(X)$, we get
$$= \Hom(\ms O_X, \RHom(\RHom(G, \O0), \omega_X)[1]) = \Hom(\RHom(G, \O0), \omega_X[1]) $$
$$= \Hom(\ms O_X, \RHom(G, \O0)[1])^* = \Hom(G, \O0[1])^*.$$

\end{proof}

Recall we are interested in a bound for  $a$ such that $H^1(X, \omega_X^{DB} \dt L^{\otimes a} \otimes I_Z) = 0.$ In light of Proposition \ref{prop:restatement}, our strategy will be to use Bridgeland stability conditions in order to show $\Hom(L^{\otimes a} \otimes I_Z, \O0[1]) = 0.$ More generally, we consider the vanishing of $\Hom(L \otimes I_Z, F)$ where $L$ is an ample line bundle, $Z \subset X$ a subscheme of dimension 0, and $F \in D^b(X)$ is of type $O$. One advantage of this more general setup is that we can obtain better bounds by relating $\Hom(L \otimes I_Z, F)$ to other spaces $\Hom(L \otimes I_{Z'}, F')$ given by changing the relative lengths of $Z$ and $\mc H^0(F)$. In what follows, we use the notation $l_Z := \Length(Z)$ and $l_F := \Length(\mc H^0(F))$, and similarly for $Z'$ and $F'$.

\begin{proposition}\label{prop:length1}
Let $L, Z, F$ be as above, with $l_F > 0$. Then if $\Hom(L \otimes I_Z, F) \ne 0$ we can find a dimension 0 subscheme $Z'$ and $F'$ of type $O$ with $\Hom(L \otimes I_{Z'}, F') \ne 0$, $l_{Z'} \le l_Z + 1$, and $l_{F'} = l_F - 1$. 
\end{proposition}

\begin{proof}
Choose a surjection $\mc H^0(F) \twoheadrightarrow \ms O_p$ and some $f \ne 0 \in \Hom(L \otimes I_Z, F)$. Then consider the map
$$g: L \otimes I_Z \xrightarrow{f} F \twoheadrightarrow \mc H^0(F) \twoheadrightarrow \ms O_p.$$
Since $L \otimes I_Z, \ms O_p$ belong to the heart $\Coh_{1/2}$, we can define $\ker(g) \in \Coh_{1/2}$; similarly, we define $F' := \ker(F \twoheadrightarrow \mc H^0(F) \twoheadrightarrow \ms O_p) \in \Coh_{1/2}$. It is immediate that $F'$ is of type $O$ with $l_{F'} = l_F - 1$, and as for $\ker(g)$, we note that either $g = 0$, in which case $\ker(g) = L \otimes I_Z$ and we take $Z' := Z$, or $g \ne 0$ and is thus surjective, in which case $\ker(g) = L \otimes I_{Z'}$ for some $Z' \supset Z$ with $l_Z' = l_Z + 1$ by the long exact sequence of cohomology.

We note that the restriction of $f$ to $L \otimes I_{Z'}$ is nontrivial (because $\Hom(\ms O_p, F) = 0$ by assumption) and factors through $F' \hookrightarrow F$, and so we find an element $f' \ne 0 \in \Hom(L \otimes I_{Z'}, F')$.
\end{proof}

\begin{proposition}\label{prop:length2}
Let $L, Z, F$ be as above, with $l_Z > 0$. Then if $\Hom(L \otimes I_Z, F) \ne 0$ we can find a dimension 0 subscheme $Z'$ and $F'$ of type $O$ with $\Hom(L \otimes I_{Z'}, F') \ne 0$, $l_{Z'} = l_Z - 1$, and $l_{F'} \le l_F + 1$.
\end{proposition}

\begin{proof}
Choose a subscheme $Z' \subset Z$ with $l_{Z'} = l_Z - 1$, and let $\ms O_p := I_{Z'}/I_Z$. Then
$$L \otimes I_Z \rightarrow L \otimes I_{Z'} \rightarrow \ms O_p \xrightarrow{+1}$$
gives a map $\Hom(L \otimes I_Z, F) \rightarrow \Hom(\ms O_p[-1], F)$. Choose $f \ne 0 \in \Hom(L \otimes I_Z, F)$, and let $g \in \Hom(\ms O_p[-1], F)$ be its image. If $g = 0$, then $f$ is the image of some $f' \ne 0 \in \Hom(L \otimes I_{Z'}, F)$, and so for $F' := F$ we get $\Hom(L \otimes I_{Z'}, F') \ne 0$.

Otherwise, let $F' := \Cone(g)$. By chasing the diagram
\[
\begin{tikzcd}
& & id \in \Hom(\ms O_p[-1], \ms O_p[-1]) \arrow[d] \\
& f \in \Hom(L \otimes I_Z, F) \arrow[r] \arrow[d] &g \in \Hom(\ms O_p[-1], F) \arrow[d] \\
f'\in\Hom(L \otimes I_{Z'}, F') \arrow[r] & f'' \in \Hom(L \otimes I_Z, F') \arrow[r] & 0 \in \Hom(\ms O_p[-1], F')
\end{tikzcd}\]
we find a morphism $f' \ne 0 \in \Hom(L \otimes I_{Z'}, F')$ whose restriction $f''$ to $L \otimes I_Z$ is given by $f$. Furthermore, we can see the cohomology groups of $F'$ are of the correct form, with $l_{F'} = l_F + 1$, and $\Hom(\ms O_q, F') = 0$ for every closed point $q$: for $q \ne p$ this is clear from the fact that $\Ext^i(\ms O_q, \ms O_p) = 0$ $\forall i$, and for $q = p$ it follows from the long exact sequence
$$0 = \Hom(\ms O_p, F) \rightarrow \Hom(\ms O_p, F') \rightarrow \Hom(\ms O_p, \ms O_p) \xhookrightarrow{g[1] \circ -} \Hom(\ms O_p, F[1]) \rightarrow \cdots$$
Thus $F'$ is of type $O$.
\end{proof}

\begin{remark}
    For the purposes of the main theorem, it is convenient to assume that decreasing $l_F$ by 1 will have the effect of increasing $l_Z$ by 1 and vice versa, while the preceding propositions leave open the possibility that in fact $l_Z$ may not need to increase if $f: L \otimes I_Z \to F$ already factors through the chosen $F' \subset F$, and similarly $l_F$ may not need to increase if $f: L \otimes I_Z \to F$ can be extended to $L \otimes I_Z'$ for the chosen $Z' \subsetneq Z$. However, we note that it is always possible to arrange for $l_Z + l_F = l_{Z'} + l_{F'}$ to remain constant: in the first case, one takes an arbitrary $Z' \supset Z$ with $l_{Z'} = l_Z + 1$ and restricts $L \otimes I_Z \to F'$ to $L \otimes I_{Z'}$, and in the second case, one chooses an arbitrary point $p \in X \setminus \supp(\mc H^0(F))$ and takes $F'$ to be the cone of an arbitrary nonzero element of $\Hom(\ms O_p[-1], F)$ and extends $L \otimes I_{Z'} \to F$ by $F \hookrightarrow F'$. (Note that by assumption on $p$, we have $\Hom(\ms O_p[-1], F) = \Hom(\ms O_p[-1], \ms O[1]) = \Hom(\DD(\ms O[2]), \DD(\ms O_p)) = \Hom(\omega_X, \ms O_p) \ne 0$.)
\end{remark}

Before proving our main results, we note that a priori our methods give the vanishing of $H^1(X, \omega_X\dt (L \otimes I_Z)) = 0,$ where the tensor products are derived. To obtain Reider-type results one would of course like to have the subjectivity of the map  $H^0(X, \omega_X \otimes L)\rightarrow H^0(X, \omega_X \otimes L\otimes \mc O_Z).$ This is implied by the following   result:

\begin{proposition}\label{prop:derivedtensor}
Let $X$ be a normal surface, $G \in \Coh(X)$ a coherent sheaf and $Z \subset X$ a 0-dimensional subscheme. Then the vanishing $H^1(X, G\dt  I_Z) = 0$ implies the surjectivity of $H^0(X, G)\rightarrow H^0(X, G\otimes \mc O_Z).$
\end{proposition}

\begin{proof}

The vanishing $H^1(X, G\dt  I_Z) = 0$ implies the map $H^0(X, G)\rightarrow H^0(X, G\dt \mc O_Z)$ is surjective.
    Consider the spectral sequence $$E_2^{p,q} = H^p(X, \mc H^q(G \dt \mc O_Z)) \Rightarrow H^{p+q}(X, G \dt \mc O_Z).$$
    % We are interested in the nonzero terms $E_2^{p,q},$ where $p+q=0.$
    Note that  $\mc H^0(G \dt \mc O_Z)=G \otimes \mc O_Z$
    %$\mc H^i(F \dt \mc O_Z)=0$ for $i>0$ and $\mc H^i(F \dt \mc O_Z)$ is supported in codimension 2 for $i<0.$
    and $\mc H^i(G \dt \mc O_Z)$ is supported in codimension 2 for all $i$. Therefore only the terms $E_2^{0,q}$ are nontrivial, so the spectral sequence degenerates and $H^0(X, G \dt O_Z) = E_2^{0,0} = H^0(X, G \otimes O_Z)$.
    % Therefore, $E_2^{0,0}=H^0(X,F\otimes \mc O_Z)$ is the only nontrivial term with $p+q=0.$ Moreover, all incoming maps will be from $E_r^{p,q}$ with $p < 0$, so trivial, and all outgoing maps will be to $E_r^{p,q}$ with $p >2$, which are again trivial.
\end{proof}

\begin{remark}
    One can similarly show $H^1(X, \omega_X \dt (L \otimes I_Z)) = 0$ implies $H^1(X, \omega_X\otimes (L \otimes I_Z)) = 0$ since the non-Gorenstein locus is codimension $2$ but we will not make use of this fact.
\end{remark}

We now come to the main technical theorem of the paper.

\begin{theorem}\label{thm: generalvanishing}
Let $X$ be a projective normal surface over an algebraically closed field, $Z \subset X$ a subscheme of length $l_Z$, and $L$ an ample line bundle on $X$ with $c_1(L)=: H$. Let $F$ be an object of type $O$ with $l_T := \Length(\mc H^0(F))$. Choose nonnegative integers $l_1, l_2$ with $l_1+l_2 = l_Z + l_T$. Then
$$\Hom(L \otimes I_Z, F) = 0$$
if $H^2 \ge \max((C_X + 2l_2+1)^2, 4(l_1+l_2+C_X)+\epsilon)$ and $H \cdot C \ge \max(2(2l_1 + C_X), C_X + 2l_2 + 1)$ for any nonzero effective divisor $C$ on $X$.
\end{theorem}

\begin{proof}
By repeated applications of Propositions \ref{prop:length1} and \ref{prop:length2}, we see that to show vanishing of $\Hom(L \otimes I_Z, F)$, it suffices show that $\Hom(L \otimes I_{Z'}, F') = 0$ for all subschemes $Z'$ of length $l_1$ and objects $F'$ of type $O$ with $\Length(\mc H^0(F'))=l_2$.

By Proposition \ref{prop:r=1}, if $f \ne 0 \in \Hom(L \otimes I_{Z'}, F')$, then $\im f=: G$ is a torsion sheaf. By Proposition \ref{prop: rk1}, we see that the quotient $L \otimes I_{Z'} \twoheadrightarrow G$ is not destabilizing at the stability condition $H := c_1(L), s := \frac{1}{2}, t_1 := \sqrt{\frac{1}{4}-\frac{2l_1+C_X}{H^2}}$
as long as $H^2 > 4(2l_1 + C_X)$ and $H \cdot \Ch_1(G) \ge 2(2l_1 + C_X)$. So we conclude in particular that $d_{s,t_1}(\im f)>0=\nu_{s,t_1}(L \otimes I_{Z'})$. Then we note that $d_{s,t}(\im f)$ is independent of $t$ since $\Ch_0(\im f) = 0$, so $d_{s,t_2}(\im f) > 0$ for $t_2 := \sqrt{\frac{1}{4}-\frac{2l_2+C_X}{H^2}}$, which implies in particular that $\im f \hookrightarrow F'$ is strictly destabilizing at $(s,t_2)$. But this is impossible if $H^2 \ge (C_X + 2l_2 + 1)^2$ and $H \cdot \Ch_1(G) \ge C_X + 2l_2+1$ by Proposition \ref{prop: Ork1}.
% we can't have l_2 = 0, C_X = 1, H^2 = 4 by H^2 > 4(l_1+l_2+C_X)
\end{proof}

In order to turn the previous theorem into a Fujita-type bound, we use the auxiliary functions $m, m'$ defined in Definition \ref{def: m'}.
% \begin{definition}\label{def: m}
%     Define a function $m: \RR_{\ge 0} \times \ZZ_{\ge 0} \to \RR$ by
%     $$(C_X, l) \mapsto \min_{\substack{l_1+l_2 = l \\ l_1, l_2 \in \ZZ_{\ge 0}}} \{\max\left(2(2l_1+C_X), C_X + 2l_2+1\right)\}.$$
%     Also define
%     $$m'(C_X, l) := \begin{cases}
%         3 & C_X = 1, l = 0 \\
%         m(C_X, l) & \textrm{else}
%     \end{cases}$$
% \end{definition}

\begin{remark}
We can also describe $m(C_X, l)$ as
$$\min\left(\max(2C_X, C_X + 2l+1), C_X + 2\left\lceil{\frac{4l+C_X+1}{6}}\right \rceil+1, 2C_X + 4\left\lceil{\frac{2l-C_X+1}{6}}\right \rceil\right)$$
which shows, for example, that $m(C_X, l) \approx \frac{4}{3} l$ for $l \gg 0$.
%Without the condition $l_1, l_2 \in \ZZ_{\ge 0}$, one finds that the minimum occurs at $l_1 = \frac{2l-C_X+1}{6}$; if this is between two positive integers, then the minimum occurs at one of those integers, and if it is negative, then the minimum is at $l_1=0$. Therefore, this minimum is
%$$ \min\left\{\max(2C_X, C_X + 2l+1), C_X + 2\left\lceil{\frac{4l+C_X+1}{6}}\right \rceil+1, 2C_X + 4\left\lceil{\frac{2l-C_X+1}{6}}\right \rceil\right\}$$
% (2l_1 + C_X)^2 >? 2l_1 - C_X: if l_1=0, then OK for C_X > 0 and if C_X = 0 then use the C_X + 2l_2 + 1 bound instead and it's fine; if l_1 \ge 1, then LHS \ge 4l_1^2 > 2l_1 so fine
% if m < l_1 coord of min < m+1, then at m, the C_X + 2l_2 + 1 is bigger; at m+1, the 2(C_X + 2l_1) is bigger
% if m < l_1 \< m+1, then n-m-1 < l_2 < n-m
% if l_1 coord of min happens at an integer, then the last two are equal and correct
In the following corollary, we also prove the lower bound $m(C_X, l) \ge C_X + l$.
\end{remark}

\begin{corollary}\label{cor: fujita}
  Let $X$ be a projective normal surface over an algebraically closed field, $Z \subset X$ a subscheme of length $l_Z$, $L$ an ample line bundle on $X$, and $F$ an object of type $O$ with $\Length(\mc H^0(F)) = l_T$. Let $l := l_Z + l_T$. Then
$$\Hom(L^{\otimes a} \otimes I_Z, F) = 0$$
whenever $a \ge m'(C_X, l)$.
\end{corollary}

\begin{proof}
Using that $(a \Ch_1(L))^2 \ge a^2$, setting $a \ge m(C_X, l)$ clearly accounts for all the bounds in Theorem \ref{thm: generalvanishing} except for $H^2 > 4(l_1+l_2+C_X)$. We claim that the only case in which this bound is relevant is $l = 0, C_X = 1$: indeed, for any choice of $l_1, l_2 \in \ZZ_{\ge 0}$ with $l_1+ l_2 = l$, we clearly have
$$\max(2(2l_1+C_X), C_X + 2l_2 + 1) \ge \max(C_X + 2l_1, C_X + 2l_2) \ge C_X + l$$
and so
$$m(C_X, l) \ge C_X + l > 2\sqrt{C_X + l}$$
for $C_X + l > 2$. It therefore suffices to consider the case $C_X + l \le 2$.

When $l = 0$, we have
$$m(C_X, 0) = \max(2C_X, C_X + 1) \ge 2 \sqrt{C_X}$$
with equality exactly when $C_X = 1$, in which case $a > 2 \sqrt{C_X}$ implies we need $a \ge 3$. When $l = 1$ and $C_X \le 1$ we compute
\begin{equation*}
    \begin{aligned}
        m(C_X, 1) &= \min(\max(2C_X, C_X + 3), \max(2C_X + 4, C_X + 1))\\ &= \min(C_X + 3, 2C_X + 4) > 2 \sqrt{C_X + 1}
    \end{aligned}
\end{equation*}

and when $l = 2$ and $C_X = 0$ we again have
$$m(0,2) = 4 > 2 \sqrt{2}.$$
\end{proof}

Finally, we also give a more explicitly ``Reider-type" statement relating the existence of a nontrivial homomorphism to the existence of an effective divisor passing through the points of interest with certain intersection numbers.

\begin{theorem}\label{thm: Reider}
Let $X$ be a projective normal surface over an algebraically closed field, $Z \subset X$ a subscheme of length $l_Z$, $L$ an ample line bundle on $X$, and $F$ an object of type $O$ with $\Length(\mc H^0(F)) = l_T$. 
Define $l' := 2\left\lfloor{\frac{l_Z + l_T+1}{2}}\right\rfloor$, i.e., $l' = l_Z + l_T$ if this is even and $l_Z + l_T + 1$ otherwise.
Assume that $H := c_1(L)$ satisfies $H^2 > (C_X + l' + 1)^2$.
Then, if
$$\Hom(L \otimes I_Z, F) \ne 0$$
there is an effective divisor $D$ such that
$D^2 < 1$ and
$$0 < D \cdot H \le D^2 + C_X + l'.$$
Moreover, if one assumes
$$\Hom(L \otimes I_{Z'}, F') = 0$$
for every pair of subscheme $Z' \subset Z$ and subobject $F' \subset F$ of type $O$ (i.e., with $\mc H^0(F') \subset \mc H^0(F')$) other than $(Z,F)$,
then $D$ passes through all the points of $\supp Z \cup \supp \mc H^0(F)$ 
\end{theorem}

\begin{proof}
    Choose $f \ne 0 \in \Hom(L \otimes I_Z, F)$. Then by Proposition \ref{prop:r=1} (whose bound is implied by our bound on $H^2$) we have that $\im F$ is a torsion sheaf, so in particular $c_1(\im f) =: D$ is an effective divisor. (Note that $\im f$ does not contain any points outside of $D$ by the assumption that there is no nontrivial map from a skyscraper sheaf to $F$; in particular, by assumption $f \ne 0$, we have that $D$ is nontrivial.) Note as well that if $D$ did not pass through some point $p \in \supp Z$, then letting $Z' \subsetneq Z$ be the subscheme supported on $\supp Z \setminus p$, we would have a map
    $$L \otimes I_Z \hookrightarrow L \otimes I_{Z'} \twoheadrightarrow \im f$$
    given by $L \otimes I_Z = L \otimes I_{Z'} \twoheadrightarrow \im F$ on $X \setminus p$ and 0 on $X \setminus \supp D$, or in other words, we would get that $f$ is the restriction of some nonzero element of $\Hom(L \otimes I_{Z'}, F)$. Similarly, if $D$ did not pass through some point $p \in \supp \mc H^0(F)$, letting
    $$F' := \ker(F \twoheadrightarrow \mc H^0(F) \twoheadrightarrow \mc H^0(F)|_p)$$
    we see that the composition
    $$\im f \hookrightarrow F \twoheadrightarrow \mc H^0(F)|_p$$
    must be trivial, and thus $\im f \hookrightarrow F$ factors through $F'$, showing that $f$ gives a nonzero element of $\Hom(L \otimes I_Z, F').$

    Now, as before, we note that the Bridgeland degree $d_{1/2,t}(\im f)$ is independent of $t$, as $\Ch_0(\im f) = 0$. So in particular, either $d_{1/2, t}(\im f) \le 0$, in which case $L \otimes I_Z \twoheadrightarrow \im f$ is a destabilizing quotient at $t = \sqrt{\frac{1}{4} - \frac{2l_Z + C_X}{H^2}}$, or $d_{1/2, t}(\im f) > 0$, in which case $\im f \hookrightarrow F$ is a strictly destabilizing subsheaf at $t = \sqrt{\frac{1}{4} - \frac{2l_T + C_X}{H^2}}$. In the first case, we see as in the proof of Proposition \ref{prop: rk1} that we must have
    $$D \cdot H \le D^2 + C_X + 2l_Z$$
    which, combined with the fact that
    $$0 < r_{1/2,t}(\im f) \le r_{1/2,t}(L \otimes I_Z) = t\frac{H^2}{2}$$
    (the rank being nonzero because $\Ch_0(\im f) = 0$ and $H \cdot D > 0$), means that we are in the situation of Lemma \ref{lemma:Hodge}, assuming $H^2 > (C_X + 2l_Z+1)^2$, and thus $D^2 <1$.
    Similarly, in the second case, as in the proof of Proposition \ref{prop: O} (where here $E = \mc H^{-1}(\coker f)$ satisfies $\Ch_0(E) = 1, \Ch_1(E) = D$) we get
    $$D \cdot H \le D^2 + C_X + 2l_T$$
    and thus again are in the situation of Lemma \ref{lemma:Hodge}, assuming $H^2 > (C_X + 2l_T + 1)^2$, and can deduce $D^2 < 1$.
    So we conclude that we must have
    $$D \cdot H \le D^2 + C_X + 2 \max(l_Z, l_T)$$
    with $D^2 < 1$.

    Finally, it remains to note that we can optimize by using Propositions \ref{prop:length1} and \ref{prop:length2} to redistribute points between $Z$ and $\mc H^0(F)$ while preserving $\supp Z \cup \supp \mc H^0(F)$, $l_Z + l_T$, and $D$ (this last point being because the different homomorphisms constructed all agree away from $\supp Z \cup \supp \mc H^0(F)$, which is of codimension 2 and so has no effect on the divisor corresponding to $\im f$).
    % Remark: what if in extending F we need to add points outside of supp Z cup supp F? This is OK: suppose we end up with f: I_Z' to F', where F' includes some extraneous points not in the support and also not in the image of F. Use 4.2 to restrict from F' to F including extraneous points (i.e., only take quotients by O_p where O_p belongs). This is definitely still defined on I_Z, so get some map I_{Z''} to (F including extraneous points) where Z' subset Z'' \subsetneq Z. But in fact this gives a map I_{Z''} to F, which is not allowed.
    In particular, we see that the lower bound for $H^2$ is maximized and the upper bound for $D \cdot H$ is minimized when $l_{Z'}, l_{T'}$ are as close to equal as possible, i.e., either both $\frac{l_Z + l_T}{2}$ or $\{\frac{l_Z + l_T + 1}{2}, \frac{l_Z + l_T - 1}{2}\}$.
\end{proof}

Our desired applications immediately follow by applying the above to $F = \O0[1]$ and $\mc O_X[1]$.

%\begin{theorem}
%Let $X$ be a projective complex normal surface, $Z \subset X$ a subscheme of length $l_Z$, and $L$ an ample line bundle on $X$. Let $l_T := \Length(\mc H^1(\O0))$. Suppose $H := \Ch_1(L)$ satisfies
%$H^2 \ge \max( (C_X + 2l_2+1)^2, 4(l_1+l_2+C_X) + \epsilon) $
%and $H \cdot C \ge \max(2(2l_1 + C_X), C_X + 2l_2 + 1)$ for a pair of nonnegative integers $l_1, l_2$ satisfying $l_Z + l_T = l_1 + l_2$.
%Then we have the vanishing
%$$H^1(X, \omega_X^{DB} \dt L \otimes I_Z) = 0.$$ In particular, this implies $\omega_X^{DB} \otimes L$ separates jets along $Z,$ i.e. 
%$$H^0(\omega_X^{DB}\otimes L)\rightarrow H^0(\omega_X^{DB}\otimes L\otimes \mc O_Z)$$ is surjective.
%\end{theorem}

\begin{proof}[Proof of Theorem \ref{thm: DB+RD}]
First of all, by Proposition \ref{prop:derivedtensor} it suffices to show $H^1(X, \omega_X^{DB} \dt L \otimes I_Z)=0.$ By Proposition \ref{prop:restatement}, vanishing of $H^1(X, \omega_X^{DB} \dt L \otimes I_Z)$ is equivalent to vanishing of $\Hom(L \otimes I_Z, \O0[1])$. The conclusion then follows by Theorem \ref{thm: generalvanishing} where $F=\O0[1] $.
\end{proof}
By letting $F=\mc O_X[1]$ in Theorem \ref{thm: generalvanishing} we can similarly deduce Theorem \ref{thm :VanishingOmega}. In particular, the classical Fujita's conjecture for surfaces follows:

\begin{proof}[Proof of Corollary \ref{cor: intro cx=0 fujita}]
    This follows by applying Corollary \ref{cor: abound} with $l_Z = 1$ and 2, respectively. In the first case, we see that for $l_1 = 0, l_2 =1$ we have $\max(2(2l_1 + C_X), C_X + 2l_2 + 1) = 3$, and in the second case, for $l_1 = 1 = l_2$, we have $\max(2(2l_1+C_X), C_X + 2l_2 + 1) = 4$.
\end{proof}

The Reider-type theorem also recovers a version of the classical Reider theorem for very ampleness (differing from \cite{reider1988vector} only in that $L$ is assumed to be ample, not nef):
\begin{corollary}\label{cor: Reidervample}
Let $X$ be a smooth projective normal surface with $C_X = 0$ (e.g. a smooth complex surface) and $L$ an ample line bundle such that $H := c_1(L)$ satisfies $H^2 > 9$. Let $Z$ be a subscheme of length 2. Then if
$$H^0(X, \omega_X \otimes L) \to H^0(X, \omega_X \otimes L \otimes O_Z)$$
is not surjective there exists an effective divisor $D$ passing through $Z$ (i.e., passing through $p$ if $\supp Z = \{p\}$, and if $\supp Z = \{p,q\}$, then $D$ passes through at least one of $p,q$, and both if neither is a base point of $\omega_X \otimes L$) such that
$$\textrm{either } D \cdot H = 1 \textrm{ and } D^2 = 0, -1 \textrm{ or } D \cdot H = 2 \textrm{ and } D^2 = 0.$$
\end{corollary}

\begin{proof}
If this is not a surjection, then we must have $H^1(X, \omega_X \otimes L \otimes I_Z) \ne 0$, which is equivalent to $\Hom(L \otimes I_Z, \ms O_X[1]) \ne 0$ by Proposition \ref{prop:restatement}. Thus, we can apply Theorem \ref{thm: Reider} with $l_Z = 2$ and $l_T = C_X = 0$ (thus $l' = 2$) to conclude that there is an effective divisor $D$ such that $D^2 < 1$ and $0 < D \cdot H \le D^2 + 2$. (For the statement about support, note that $H^0(X, \omega_X \otimes L) = 0$ by Theorem \ref{thm: generalvanishing}, so $D$ must at least pass through some point of $Z$.)
Since $X$ was assumed to be smooth, we have $D^2 \in \ZZ$ (and the same, of course, is true for $D \cdot H$), so the statement follows.
\end{proof}

\printbibliography
\Addresses

\end{document}